\documentclass[reqno,b5paper]{amsart}
\usepackage{amsmath}
\usepackage{amssymb}
\usepackage{amsthm}
\usepackage{enumerate}
\usepackage[mathscr]{eucal}
\theoremstyle{plain}
\newtheorem{thm}{Theorem}[section]
\newtheorem{cor}[thm]{Corollary}
\newtheorem{lem}[thm]{Lemma}

\newtheorem{note}{Note}[section]

\theoremstyle{definition}
\newtheorem{defn}{Definition}[section]
\newtheorem{rem}{Remark}[section]

\begin{document}

\setcounter {page}{1}
\title{On strong $\mathcal{A}$-statistical convergence in probabilistic metric spaces}
\author[ P. Malik, S. Das ]{Prasanta Malik*, and Samiran Das*\ }
\newcommand{\acr}{\newline\indent}
\maketitle
\address{{*\,} Department of Mathematics, The University of Burdwan, Golapbag, Burdwan-713104,
West Bengal, India.
                Email: pmjupm@yahoo.co.in, das91samiran@gmail.com \acr
           }

\maketitle
%---------- Abstract ----------
\begin{abstract}
In this paper we study some basic properties of strong
$\mathcal{A}$-statistical convergence and strong
$\mathcal{A}$-statistical Cauchyness of sequences in probabilistic
metric spaces not done earlier. We also study some basic
properties of strong $\mathcal{A}$-statistical limit points and
strong $\mathcal{A}$-statistical cluster points of a sequence in a
probabilistic metric space. Further we also introduce the notion of
strong statistically $\mathcal{A}$-summable sequence in a
probabilistic metric space and study its relationship with strong
$\mathcal{A}$-statistical convergence.
\end{abstract}

\author{}
\maketitle { Key words and phrases: probabilistic metric space, strong $\mathcal{A}$-statistical convergence,
strong $\mathcal{A}$-statistical Cauchy sequence, strong statistically $\mathcal{A}$-summable sequence,
strong $\mathcal{A}$-statistical limit point, strong $\mathcal{A}$-statistical cluster point.} \\

\textbf {AMS subject classification (2010) : 54E70, 40C05}.  \\

%---------- 1. Section ----------
\section{\textbf{Introduction and background:}}
The concept of probabilistic metric (PM) was introduced by K.
Menger \cite{Me} under the name of ``statistical metric" by
considering the distance between two points $a$ and $b$ as a
distribution function $\mathcal{F}_{ab}$ instead of a non-negative
real number and the value of the function $\mathcal{F}_{ab}$ at
any $t> 0 $ i.e. $\mathcal{F}_{ab} (t) $ is interpreted as the
probability that the distance between the points $a$ and $b$ is
less than $t$. After Menger, a through development by Schwiezer
and Sklar \cite{Sh1, Sh2, Sh3, Sh4}, Tardiff \cite{Tar}, Thorp
\cite{Th} and many others brought this theory of probabilistic
metric to its present state. A detailed study on probabilistic
metric spaces can be found in the famous book of Schwiezer and
Sklar \cite{Sh5}. Several topologies are defined on a PM space,
but strong topology is one of them, which received most attention
to date and it is the main tool of our paper.

The notion of statistical convergence was introduced as a
generalization to the usual notion of convergence of real
sequences independently by Fast \cite{Fa} and Schoenberg \cite{Sc}
using the concept of natural density of subsets of $\mathbb{N}$,
the set of all nonzero positive integers. A set
$\mathcal{K}\subset\mathbb{N}$ has natural density
$d(\mathcal{K})$ if
$$d(\mathcal{K})=\lim\limits_{n\rightarrow \infty}\frac{\left|\mathcal{K}(n)\right|}{n}$$
where $\mathcal{K}(n)=\left\{j\in \mathcal{K}:j\leq n\right\}$ and
$\left|\mathcal{K}(n)\right|$ represents the number of elements in
$\mathcal{K}(n)$. Note that one can write
$d(\mathcal{K})=\lim\limits_{n\rightarrow \infty}(C_1
\chi_{_\mathcal K})_n$, where $\chi_{_\mathcal{K}}$ is the
characteristic function of $\mathcal{K}$ and $C_1$ is the Cesaro
matrix i.e. $C_1 = (a_{nk})$ is an $\mathbb{N} \times \mathbb{N}$
matrix given by $a_{nk} = \frac{1}{n}$ if $ n \geq k$ and $a_{nk}
= 0$ if $ n < k$.

A sequence $x=\{x_k\}_{k\in\mathbb{N}}$ of real numbers is called
statistically convergent to $\mathcal{L}\in\mathbb{R}$ if for
every
$\varepsilon>0,~d(\{k\in\mathbb{N}:\left|x_k-\mathcal{L}\right|\geq\varepsilon\})=0$.

Study in this line turned out to be one of the active research
area in summability theory after the works of Salat \cite{Sa} and
Fridy \cite{Fr1}. For more works in this direction one can see
\cite{Fr2, Pe, St}.

The notion of statistical convergence was further generalized to
$\mathcal{I}$-convergence by Kostyrko et al. \cite{Ko1} based on
the notion of an ideal $\mathcal I$ of subsets of $\mathbb{N}$.

A non-empty family $\mathcal I$ of subsets of a non empty set $S$
is called an ideal in $S$ if $\mathcal I$ is hereditary ( i.e. $ A
\in \mathcal{I},~ B\subset A \Rightarrow B\in \mathcal{I}$ ) and
additive ( i.e. $ A,B \in \mathcal{I} \Rightarrow A \cup B\in
\mathcal{I}$).

An ideal $\mathcal{I}$ in a non-empty set $S$ is called
non-trivial if $S\notin \mathcal{I}$ and $\mathcal{I} \neq
\{\emptyset\}$.

A non-trivial ideal $\mathcal{I}$ in $S(\neq \emptyset)$ is called
admissible if $\{z\} \in \mathcal{I}$ for each $z \in S$.

%Throughout the paper we take $\mathcal{I}$ as a non-trivial
%admissible ideal in $\mathbb{N}$ unless otherwise mentioned.

Let $\mathcal I$ be an admissible ideal in $\mathbb{N}$. A
sequence $x=\{x_k\}_{k \in \mathbb{N}}$ of real numbers is said to
be $\mathcal{I}$-convergent to $\xi$ if for any $\varepsilon
>0,~ \{ k\in\mathbb{N} : \left| x_k - \xi \right| \geq  \varepsilon\} \in \mathcal{I}.$
In this case we write $
\mathcal{I}\mbox{-}\lim\limits_{k\rightarrow \infty}x_k = \xi$.

More works in this line can be seen in \cite{De2, Ko2, La1, La2}
and many others.

In 1981, Freedman and Sember \cite{Fd} generalized the concept of
natural density to the notion of $\mathcal A$-density by replacing
the Cesaro matrix $C_1$ with an arbitrary nonnegative regular
summability matrix $\mathcal A$. An $\mathbb{N} \times \mathbb{N}$
matrix $\mathcal A = (a_{nk}),~ a_{nk} \in \mathbb{R}$ is said to
be a regular summability matrix if for any convergent sequence
$x=\{x_k\}_{k\in\mathbb N}$ with limit $\xi$, $
\displaystyle{\lim_{n \rightarrow \infty}} \displaystyle{\sum_{ k
= 1}^{\infty}} a_{nk}x_k = \xi$, and $\mathcal A$ is called
nonnegative if $ a_{nk} \geq 0, \forall n, k $. The well-known
Silvermann- Toepliz's theorem asserts that an $\mathbb{N} \times
\mathbb{N}$ matrix $\mathcal A = (a_{nk}),~ a_{nk} \in \mathbb{R}$
is regular if and only if the following three conditions are
satisfied:

\begin{enumerate}[(i)]
\item
$\left\|\mathcal{A}\right\|=\sup\limits_{n}\sum\limits_{k}|a_{nk}|<\infty$,
\item $\lim\limits_{n\rightarrow\infty}a_{nk}=0$ for each $k$,
\item $\lim\limits_{n\rightarrow\infty}\sum\limits_{k}a_{nk}=1$.
\end{enumerate}

Throughout the paper we take $\mathcal A = (a_{nk})$ as an
$\mathbb{N} \times \mathbb{N}$ non negative regular summability
matrix.

For a non negative regular summability matrix $ \mathcal A =
(a_{nk})$, a set $ \mathcal{B} \subset \mathbb{N}$ is said to have
$\mathcal A$-density $\delta_{\mathcal A} (\mathcal{B})$, if
$$\delta_{\mathcal A} (\mathcal{B})=\lim\limits_{n\rightarrow \infty}\displaystyle{\sum_{ k \in \mathcal{B}}} a_{nk}.$$
It is clear that if $B,C\subset \mathbb{N}$ with
$\delta_\mathcal{A}(B)=0$ and $\delta_\mathcal{A}(C)=0$ then
$\delta_\mathcal{A}(B\cup C)=0$. Also $\delta_\mathcal{A}(B^c)=1$.
Note that for the Cesaro matrix $C_1$, $C_1$-density become
natural density.

If a real sequence $x=\{x_k\}_{k\in \mathbb{N}}$ satisfies a
property $\mathcal{P}$ for each $k$ except for a set of
$\mathcal{A}$-density zero, then we say $x$ satisfies the property
$\mathcal{P}$ for ``almost all $k(\mathcal{A})$'' and we write it
in short as ``$a.a.k(\mathcal{A})$''.

Using the notion of $\mathcal A$-density, the notion of
statistical convergence was extended to the notion of $\mathcal
A$-statistical convergence by Kolk \cite{Kl1}, which included the
ideas of statistical convergence \cite {Fa, Sc},
$\lambda$-statistical convergence \cite {Mu1} or lacunary
statistical convergence \cite {Fr4} as special cases.

A sequence $x=\{x_k\}_{k \in \mathbb{N}}$ of real numbers is said
to be $\mathcal A$-statistically convergent to $\xi $ if for every
$\epsilon
> 0$, $\delta_{\mathcal A}(B(\epsilon)) = 0$, where $B(\epsilon) = \{ k\in\mathbb{N} : | x_k - \xi| \geq  \epsilon\}$.
In this case we write
$st_{\mathcal{A}}\mbox{-}\lim\limits_{k\rightarrow \infty}x_k =
L$.

A sequence $x=\{x_k\}_{k\in\mathbb{N}}$ of real numbers is said to
be $\mathcal{A}$-statistically Cauchy if for every $\gamma>0$,
there exists a natural number $k_0$ such that
$$\delta_\mathcal{A}(\{k\in\mathbb{N}: |x_k-x_{k_0}|\geq\gamma\})=0.$$

Using this notion of $A$-statistical convergence, the concepts of
statistical limit point and statistical cluster point \cite{Fr2}
of real sequences were extended to the notions of $A$-statistical
limit point and $A$-statistical cluster point by Connor et al.
\cite{Co4}.

If $\{x\}_\mathcal{Q}$ is a subsequence of a sequence
$x=\{x_k\}_{k\in\mathbb{N}}$ and $\delta_{\mathcal A}(\mathcal{Q})
= 0$, then $\{x\}_\mathcal{Q}$ is called an $\mathcal A$-thin
subsequence of $x$. On the other hand $\{x\}_\mathcal{Q}$ is
called an $\mathcal A$-nonthin subsequence of $x$ if
$\delta_{\mathcal A}(\mathcal{Q}) \neq 0$, where $\delta_{\mathcal
A}(\mathcal{Q}) \neq 0$ means that either $\delta_{\mathcal
A}(\mathcal{Q})$ is a positive number or $\mathcal{Q}$ fails to
have $\mathcal A$-density.

A real number $p$ is called an $\mathcal A$-statistical limit
point of a real sequence $x=\{x_k\}_{k\in\mathbb{N}}$, if there
exists an $\mathcal A$-non-thin subsequence of $x$ that converges
to $p$.

A real number $q$ is called an $\mathcal A$-statistical cluster
point of a real sequence $x=\{x_k\}_{k\in\mathbb N}$, if for every
$\epsilon>0$ the set $\{k \in\mathbb{N}: \left|x_k - q \right|
<\epsilon\}$ does not have $\mathcal A$-density zero.

If $\Lambda_x^\mathcal{A}$, $\Gamma_x^\mathcal{A}$ and $L_x$
denote the set of all $\mathcal{A}$-statistical limit points, the
set of all $\mathcal{A}$-statistical cluster points and the set of
all ordinary limit points of $x$, then clearly
$\Lambda_x^\mathcal{A}\subset\Gamma_x^\mathcal{A}\subset L_x$

More primary works on this convergence can be found in \cite{De1,
De3, Gu1}, where many more references are mentioned.

Because of immense importance of probabilistic metric space in
applied mathematics, the notion of statistical convergence \cite
{Fa, Sc} and the notion of $I$-convergence \cite{Ko1} were
extended to the setting of sequences in a PM space endowed with
the strong topology by \c{S}en\c{c}imen et al. in \cite{Se} and
\cite{Se2} respectively. Following Kostyrko et al. \cite{Ko1},
Bartoszewicz et al. \cite{Br} and \c{S}en\c{c}imen et al
\cite{Se2}, if we take the admissible ideal $\mathcal I$ of
subsets of $\mathbb{N}$ given by $ \mathcal I =
\mathcal{I}_{\mathcal A} = \{ B \subset \mathbb{N}:
\delta_{\mathcal A} (B) = 0\}$, where $\mathcal A = (a_{nk})$ is
an $\mathbb{N} \times \mathbb{N}$ non negative regular summability
matrix, then the notions of strong $\mathcal{I}_{\mathcal
A}$-convergence, strong $\mathcal{I}_{\mathcal A}$-Cauchyness,
strong $\mathcal{I}_{\mathcal A}$-limit point and strong
$\mathcal{I}_{\mathcal A}$-cluster point of sequences in a PM
space become the notions of strong $\mathcal A$-statistical
convergence, strong $\mathcal A$-statistical Cauchyness, strong
$\mathcal A$-statistical limit point and strong $\mathcal
A$-statistical cluster point respectively.

In this paper we study some basic properties of strong
$\mathcal{A}$-statistical convergence, strong
$\mathcal{A}$-statistical Cauchyness, strong
$\mathcal{A}$-statistical limit points and strong
$\mathcal{A}$-statistical cluster points of a sequence in a
probabilistic metric space not done earlier. Also in section 5
introducing the notion of a strong statistically
$\mathcal{A}$-summable sequence in a probabilistic metric space we
study its basic properties including its relationship with strong
$\mathcal{A}$-statistical convergence.

%------------------------------Section-2 - Basic definitions-------------------------
\section{\textbf{Basic Definitions and Notations}}
In this section we recall some preliminary concepts and results
related to probabilistic metric (PM) spaces (for more details see
\cite{Sh1, Sh2, Sh3, Sh4, Sh5}).
%---------------------------Definition 2.1------------------------------------------
\begin{defn}\cite{Sh5}
A non decreasing function $f: [-\infty, \infty] \rightarrow [0,1]$
is called a distribution function if $f(-\infty)=0$ and $
f(\infty)=1$.
\end{defn}
We denote the set of all distribution functions that are left
continuous over $(-\infty,\infty)$ by $\mathcal{D}$ and a relation
$\leq$ on $\mathcal{D}$ defined by $f \leq g$ if and only if
$f(a)\leq g(a),~\forall~ a \in [-\infty, \infty]$ is clearly a
partial order relation on $\mathcal{D}$.
% --------------------------Definition 2.2-------------------------------------------
\begin{defn}\cite{Sh5}
For any $q\in [-\infty, \infty ]$ the unit step at $q$ is denoted
by $\varepsilon_{q}$ and is defined to be a function in
$\mathcal{D}$ given by
\[ \varepsilon_{q}(x) = \left\{
\begin{array}{l l}
0, & \quad x\in [-\infty,q]\\
1, & \quad x\in (q,+\infty], ~\text{for}~ q\in[-\infty,+\infty),
\end{array} \right.\]\\
\[ \varepsilon_{\infty}(x) = \left\{
\begin{array}{l l}
0, & \quad x\in [-\infty,\infty) \\
1, & \quad x = \infty .
\end{array} \right.\]\\
\end{defn}

%-----------------------------Definition 2.3------------------------------------------
\begin{defn}\cite{Sh5}
A sequence $\{f_{n}\}_{n\in \mathbb{N}}$ of distribution functions
is said to converge weakly to a distribution function $f$, if the
sequence $\{f_{n}(x)\}_{n\in\mathbb{N}}$ converges to $f(x)$ at
each continuity point $x$ of $f$. We write
$f_{n}\xrightarrow{w}f$.
\end{defn}

% ----------------------------Definition 2.4-------------------------------------------
\begin{defn}\cite{Sh5}
If $f,g\in \mathcal{D}$, then the distance $d_{L}(f,g)$ between
$f$ and $g$ is defined as the infimum of all numbers $a\in(0,1]$
such that
\begin{eqnarray*}
&&~~~f(\xi-a)-a\leq g(\xi)\leq f(\xi+a)+a \\
&\text{and} &~g(\xi-a)-a\leq f(\xi)\leq
g(\xi+a)+a,\hspace{0.1in}\text{hold}~\text{for
all}~\xi\in\left(-\frac{1}{a},\frac{1}{a}\right).
\end{eqnarray*}
\end{defn}
Then $(\mathcal{D},d_L)$ forms a metric space with the metric $d_{L}$. Clearly
if $\{f_k\}_{k\in\mathbb{N}}$ is a sequence in $\mathcal{D}$ and
$f\in\mathcal{D}$, then $f_k\xrightarrow{w}f ~\text{if and only if}~ d_{L}(f_k,f)\rightarrow 0$.

%---------------------------------Definition 2.5--------------------------------
\begin{defn}\cite{Sh5}
A non decreasing function $f:[0,\infty]\longrightarrow\mathbb{R}$,
which is left continuous on $(0,\infty)$ is said to be a distance
distribution function if $f(0)=0$ and $ f(\infty)=1$.
\end{defn}
We denote the set consisting of all the distance distribution functions as
$\mathcal{D}^{+}$. Clearly $(\mathcal{D}^{+},d_L)$ is a compact metric space and thus complete.
% ---------------------------------Theorem 2.1------------------------------------
\begin{thm}\cite{Sh5}\label{thm61}
If $f\in\mathcal{D}^{+}$ then for any $t>0$, $f(t)>1-t$ if and only if $d_{L}(f,\varepsilon_{0})<t$.
\end{thm}

%--------------------------------------Definition2.6-------------------------------
\begin{defn}\cite{Sh5}
A triangle function is a binary operation $\tau$ on
$\mathcal{D}^{+}$, which is associative, commutative,
nondecreasing in each place and has $\varepsilon_{0}$ as the
identity element.
\end{defn}
%--------------------------------------Definition2.7--------------------------------
\begin{defn}\cite{Sh5}
A probabilistic metric space, in short PM space, is a triplet
$(X,\mathcal{F},\tau)$, where $X$ is a nonempty set whose elements
are the points of the space, $\mathcal{F}$ is a function from
$X\times X$ into $\mathcal{D}^{+}$, $\tau$ is a triangle function
and the following conditions are satisfied for all $a,b,c\in X$:
    \begin{enumerate}[(P-1).]
    \item $\mathcal{F}(a,a)=\varepsilon_{0}$
    \item $\mathcal{F}(a,b)\neq\varepsilon_{0}$ if $a\neq b$
    \item $\mathcal{F}(a,b)=\mathcal{F}(b,a) $
    \item $\mathcal{F}(a,c)\geq \tau(\mathcal{F}(a,b),\mathcal{F}(b,c))$.
    \end{enumerate}
Henceforth we will denote $\mathcal{F}(a,b)$ by $\mathcal{F}_{ab}$
and its value at $t>0$ by $\mathcal{F}_{ab}(t)$.
\end{defn}

%-------------------------------------Definition 2.8-------------------------------------------------------
\begin{defn}\cite{Sh5}
Let $(X,\mathcal{F},\tau)$ be a PM space. For $\xi\in X$ and $t>0$, the strong $t$-neighborhood of
$\xi$ is denoted by $\mathcal{N}_{\xi}(t)$ and is defined by
\begin{center}
$\mathcal{N}_{\xi}(t)=\{\eta\in X : \mathcal{F}_{\xi\eta}(t)>1-t\} $.
\end{center}
The collection $\mathfrak{N}_{\xi}=\{\mathcal{N}_{\xi}(t):t>0 \}$ is called the strong
neighborhood system at $\xi$ and the union $\mathfrak{N}=\bigcup\limits_{\xi\in X}\mathfrak{N}_{\xi}$ is
called the strong neighborhood system for $X$.
\end{defn}

From Theorem \ref{thm61}, we can write
$\mathcal{N}_{\xi}(t)=\{\eta\in X:
d_{L}(\mathcal{F}_{\xi\eta},\varepsilon_{0})<t\} $. If $\tau$ is
continuous, then the strong neighborhood system $\mathfrak{N}$
determines a Hausdorff topology for $X$. This topology is called
the strong topology for $X$ and members of this topology are
called strongly open sets.

Throughout the paper, in a PM space $(X,\mathcal{F},\tau)$, we always consider that $\tau$ is
continuous and $X$ is endowed with the strong topology.

In a PM space $(X,\mathcal{F},\tau)$ the strong closure of any subset $\mathcal{M}$ of $X$ is
denoted by $k(\mathcal{M})$ and for any subset $\mathcal{M}(\neq\emptyset)$ of $X$ strong
closure of $\mathcal{M}$ is defined by,
$$k(\mathcal{M})=\{c\in X: ~\text{for any}~ t>0, ~\text{there exists}~ e\in \mathcal{M} ~\text{such that}~\mathcal{F}_{ce}(t)>1-t\}.$$

\begin{defn}\cite{Du1}
Let $(X,\mathcal{F},\tau)$ be a PM space. Then a subset $\mathcal{M}$ of $X$ is called
strongly closed if its complement is a strongly open set.
\end{defn}
%---------------------------Definition 2.11------------------------------------------
\begin{defn}\cite{Ma5, Sh5}
Let $(X,\mathcal{F},\tau)$ be a PM space and $\mathcal{M}\neq \emptyset$ be a subset of $X$. Then
$l\in X$ is said to be a strong limit point of $\mathcal{M}$ if for every $t>0$,
$$\mathcal{N}_l(t)\cap(\mathcal{M}\setminus\{l\})\neq\emptyset.$$
The set of all strong limit points of the set $\mathcal{M}$ is denoted by $L_\mathcal{M}^\mathcal{F}$.
\end{defn}
%---------------------------Definition 2.12------------------------------------------
\begin{defn}\cite{Du1}
Let $(X,\mathcal{F},\tau)$ be a PM space and $\mathcal{M}$ be a subset of $X$. Let $\mathcal{Q}$ be
a family of strongly open subsets of $X$ such that $\mathcal{Q}$ covers $\mathcal{M}$. Then
$\mathcal{Q}$ is said to be a strong open cover for $\mathcal{M}$.
\end{defn}
%---------------------------Definition 2.13------------------------------------------
\begin{defn}\cite{Du1}
Let $(X,\mathcal{F},\tau)$ be a PM space and $\mathcal{M}$ be a subset of $X$. Then $\mathcal{M}$ is
called a strongly compact set if for every strong open cover of $\mathcal{M}$ has a finite subcover.
\end{defn}
%---------------------------Definition 2.14------------------------------------------
\begin{defn}\cite{Du1}
Let $(X,\mathcal{F},\tau)$ be a PM space, $x=\{x_k\}_{k\in\mathbb{N}}$ be a sequence in $X$. Then $x$ is
said to be strongly bounded if there exists a strongly compact subset $\mathcal{C}$ of $X$ such
that $x_k\in \mathcal{C}$, for all $k\in\mathbb{N}$.
\end{defn}
%---------------------------Theorem 2.3------------------------------------------
\begin{thm}\cite{Du1}\label{thm62}
Let $(X,\mathcal{F},\tau)$ be a PM space and $\mathcal{M}$ be a strongly compact subset of $X$. Then
every strongly closed subset of $\mathcal{M}$ is strongly compact.
\end{thm}

%--------------------------------Definition2.9--------------------------------------------------------------
\begin{defn}\cite{Sh5}
Let $(X,\mathcal{F},\tau)$ be a PM space. Then for any $u>0$, the subset $\mathcal{V}(u)$ of $ X\times X$ given by
\begin{center}
$\mathcal{V}(u)=\{(p,q):\mathcal{F}_{pq}(u)>1-u\} $
\end{center}
is called the strong $u$-vicinity.
\end{defn}
%-------------------- -------------Theorem 2.2 --------------------------------------------------------
\begin{thm}\cite{Sh5}\label{thm63}
 Let $(X,\mathcal{F},\tau)$ be a PM space and $\tau $ be continuous. Then for any $u>0$, there is an $\alpha>0$ such that
 $\mathcal{V}(\alpha)\circ\mathcal{V}(\alpha)\subset
 \mathcal{V}(u)$, where $\mathcal{V}(\alpha)\circ\mathcal{V}(\alpha)=\{(p,r):$ ~\mbox{for some}~ $q$,~  $(p,q)$
 and $(q,r)\in \mathcal{V}(\alpha)\}$.
 \end{thm}
%-------------------------------------Note 2.1---------------------------------------------------------
\begin{note}\label{no61}
From the hypothesis of Theorem \ref{thm63}, we can say that for any $u>0$,
there is an $\alpha >0 $ such that $\mathcal{F}_{pr}(u)>1-u$
 whenever $\mathcal{F}_{pq}(\alpha)>1-\alpha $ and $\mathcal{F}_{qr}(\alpha)>1-\alpha$. Equivalently it
 can be written as: for any $u>0$, there is an $\alpha>0$ such that
 $d_{L}(\mathcal{F}_{pr},\varepsilon_{0})<u$ whenever $d_{L}(\mathcal{F}_{pq},\varepsilon_{0})<\alpha$
 and $ d_{L}(\mathcal{F}_{qr}, \varepsilon_{0})<\alpha$.
\end{note}

%------------- definition 2.13.-------------
\begin{defn}\cite{Sh5}
Let $(X,\mathcal{F},\tau)$ be a PM space. A sequence $x=\{x_k\}_{k\in \mathbb{N}}$ in $X$ is said to be strongly convergent to
$\mathcal{L}\in X$ if for every $t>0$, $\exists$ a natural number $k_0$ such that
$$x_k\in\mathcal{N}_\mathcal{L}(t),\hspace{1 in} \text{whenever}~ k\geq k_0 .$$
\end{defn}
In this case we write
$\mathcal{F}$-$\lim\limits_{k\rightarrow\infty}x_k=\mathcal{L}$ or, $x_k\stackrel{\mathcal{F}}\longrightarrow \mathcal{L}$.

\begin{defn}\cite{Sh5}
Let $(X,\mathcal{F},\tau)$ be a PM space. A sequence $x=\{x_k\}_{k\in\mathbb{N}}$ in $X$ is said
to be strong Cauchy if for every $t>0$, $\exists$ a natural number $k_0$ such that
$$(x_k,x_r)\in\mathcal{U}(t), \hspace{1 in} \text{whenever}~ k,r\geq k_0.$$
\end{defn}

\section{\textbf{Strong $\mathcal{A}$-statistical convergence and strong $\mathcal{A}$-statistical Cauchyness}}

%------------- ------------------------------definition 3.1.-----------------------------------------------------------------------------------------------------------------
In this section we study some basic properties of strong
$\mathcal{A}$-statistical convergence and strong
$\mathcal{A}$-statistical Cauchyness in a PM space not studied
earlier.
\begin{defn}\cite{Se2}
Let $(X,\mathcal{F},\tau)$ be a PM space and $\mathcal I$ be an
admissible ideal in $\mathbb{N}$. A sequence $x=\{x_k\}_{k \in
\mathbb{N}}$ in $X$ is said to be strongly $\mathcal I$-convergent
to $\alpha\in X$, if for any $t> 0$
\begin{center}
$\{k\in\mathbb{N}: \mathcal{F}_{x_k\alpha}(t)\leq 1-t\} \in
\mathcal I ~~~~~~~~or~~~~~~~~ \{
k\in\mathbb{N}:x_k\notin\mathcal{N}_\alpha(t)\} \in \mathcal I.$
\end{center}
In this case we write $\mathcal
I^{\mathcal{F}}$-$\lim\limits_{k\rightarrow\infty}x_k = \alpha$.
\end{defn}

\begin{defn}\cite{Se2}
Let $(X,\mathcal{F},\tau)$ be a PM space and $\mathcal I$ be an
admissible ideal in $\mathbb{N}$. A sequence $x=\{x_k\}_{k \in
\mathbb{N}}$ in $X$ is said to be strongly $\mathcal I$-Cauchy if
for any $t>0$, $\exists$ a natural number $N_0=N_0(t)$ such that
$$\{k\in\mathbb{N}:\mathcal{F}_{x_kx_{N_0}}(t)\leq 1-t\} \in \mathcal I ~~~~~~~or~~~~~~~~ \{
k\in\mathbb{N}:x_k\notin\mathcal{N}_{N_0}(t)\} \in \mathcal I.$$
\end{defn}

\begin{note}
(i) If $\mathcal{I}=\mathcal{I}_{fin}=\{\mathcal{K}\subset
\mathbb{N}: \left|\mathcal{K}\right|<\infty\}$, then in a PM space
the notions of strong $\mathcal {I}_{fin}$-convergence and strong
$\mathcal {I}_{fin}$-Cauchyness  coincide with the notions of
strong convergence and strong Cauchyness respectively.

(ii) If $ \mathcal I = \mathcal{I}_{\mathcal A} = \{ B \subset
\mathbb{N}: \delta_{\mathcal A} (B) = 0\}$, where $\mathcal A =
(a_{nk})$ is an $\mathbb{N} \times \mathbb{N}$ non negative
regular summability matrix, then the notions of strong
$\mathcal{I}_{\mathcal A}$-convergence and strong
$\mathcal{I}_{\mathcal A}$-Cauchyness of sequences in a PM space
coincide with the notions of strong $\mathcal A$-statistical
convergence and strong $\mathcal A$-statistical Cauchyness
respectively. Further, if $\mathcal A$ be the Cesaro matrix $C_1$,
then the notions of strong $\mathcal{I}_{C_1}$-convergence and
strong $\mathcal{I}_{C_1}$-Cauchyness of sequences in a PM space
coincide with the notions of strong statistical convergence and
strong statistical Cauchyness respectively.
\end{note}

If a sequence $x=\{x_k\}_{k \in\mathbb{N}}$ in a PM space
$(X,\mathcal{F},\tau)$ is strongly $\mathcal{A}$-statistically
convergent to $\mathcal{L} \in X$, then we write,
$st_\mathcal{A}^{\mathcal{F}}$-$\lim\limits_{k\rightarrow
\infty}x_k =\mathcal{L} $ or simply as
$x_k\xrightarrow{st^\mathcal{F}_{\mathcal{A}}}\mathcal{L}$ and
$\mathcal{L}$ is called the $\mathcal{A}$-statistical limit of
$x$.

\begin{rem}\label{rem61}
The following three statements are equivalent:
\begin{enumerate}[(i)]
    \item $x_k\xrightarrow{st_\mathcal{A}^{\mathcal{F}}}\mathcal{L}$
    \item For each $t>0,~\delta_\mathcal{A}(\{k\in \mathbb{N}: d_L(\mathcal{F}_{x_k\mathcal{L}},\varepsilon_0)\geq t\})=0$
    \item $st_\mathcal{A}\mbox{-}\lim\limits_{k\rightarrow\infty}d_L(\mathcal{F}_{x_k\mathcal{L}},\varepsilon_0)=0$.
\end{enumerate}
\end{rem}
\begin{proof}
It is clear from Theorem \ref{thm61}, Definition 3.1 and Note
3.1(ii).
\end{proof}

\begin{thm}\label{thm64}
Let $(X,\mathcal{F},\tau)$ be a PM space and $x=\{x_k\}_{k\in\mathbb{N}}$ be a strongly $\mathcal{A}$-statistically
convergent sequence in $X$. Then strong $\mathcal{A}$-statistical limit of $x$ is unique.
\end{thm}

\begin{proof}
If possible, let $st_\mathcal{A}^{\mathcal{F}}$-$\lim\limits_{k\rightarrow\infty}x_k = \alpha_1 $
and $st_\mathcal{A}^{\mathcal{F}}$-$\lim\limits_{k\rightarrow \infty}x_k = \alpha_2$
with $\alpha_1\neq \alpha_2$. So $\mathcal{F}_{\alpha_1\alpha_2}\neq\varepsilon_0$. Then
there is a $t>0$ such that $d_L(\mathcal{F}_{\alpha_1\alpha_2},\varepsilon_0)=t$. We
choose $\eta>0$ so that $d_L(\mathcal{F}_{pq},\varepsilon_0)<\eta$ and $d_L(\mathcal{F}_{qr},\varepsilon_0)<\eta $ imply
that $d_L(\mathcal{F}_{pr},\varepsilon_0)<t$. Since
$st_\mathcal{A}^{\mathcal{F}}$-$\lim\limits_{k\rightarrow \infty}x_k = \alpha_1 $ and
$st_\mathcal{A}^{\mathcal{F}}$-$\lim\limits_{k\rightarrow\infty}x_k = \alpha_2$, so
$\delta_\mathcal{A}(E_1(\eta)) = 0$ and $\delta_\mathcal{A}(E_2(\eta)) = 0$, where
$$E_1(\eta)=\{k\in\mathbb{N}: \mathcal{F}_{x_k\alpha_1}(\eta)\leq 1- \eta\}$$ and
$$E_2(\eta)=\{k\in \mathbb{N}: \mathcal{F}_{x_k\alpha_2}(\eta)\leq 1- \eta\}.$$
Now let $E_3(\eta)=E_1(\eta)\cup E_2(\eta)$. Then $\delta_\mathcal{A}(E_3(\eta))=0$ and this gives
$\delta_\mathcal{A}(E^c_3(\eta))=1$. Let $k\in E^c_3(\eta).$ Then
$d_L(\mathcal{F}_{x_k\alpha_1},\varepsilon_0)<\eta$ and
$d_L(\mathcal{F}_{\alpha_2x_k},\varepsilon_0)<\eta$ and so
$d_L(\mathcal{F}_{\alpha_1\alpha_2},\varepsilon_0)<t$, this gives a contradiction. Hence
strong $\mathcal{A}$-statistical limit of a strongly $\mathcal{A}$-statistically convergent sequence in a PM space is unique.
\end{proof}

\begin{thm}\label{thm65}
Let $(X,\mathcal{F},\tau)$ be a PM space. Let
$x=\{x_k\}_{k\in\mathbb{N}}$ and $y=\{y_k\}_{k\in\mathbb{N}}$ be
two sequences in $X$ such that
$x_k\xrightarrow{st_\mathcal{A}^{\mathcal{F}}} p \in X $ and
$y_k\xrightarrow{st_\mathcal{A}^{\mathcal{F}}} q \in X$. Then
$$st_\mathcal{A}\mbox{-}\lim\limits_{k\rightarrow\infty}d_L(\mathcal{F}_{x_ky_k}, \mathcal{F}_{pq})=0.$$
\end{thm}
\begin{proof}
The proof directly follows from Theorem 3.1 \cite{Se2}, by taking
$\mathcal I = \mathcal{I}_{\mathcal A}$.
\end{proof}

\begin{thm}\label{thm66}
Let $(X,\mathcal{F},\tau)$ be a  PM space and $x=\{x_k\}_{k\in\mathbb{N}}$ be a sequence in $X$. If
$x$ is strongly convergent to $\mathcal{L} \in X$, then $st_\mathcal{A}^{\mathcal{F}}\mbox{-}\lim\limits_{k\rightarrow\infty} x_k=\mathcal{L}$.
\end{thm}
\begin{proof}
The proof is very easy, so omitted.
\end{proof}

\begin{thm}\label{thm67}
Let $(X,\mathcal{F},\tau)$ be a PM space and $x=\{x_k\}_{k\in\mathbb{N}}$ be a sequence in $X$. Then
$st_\mathcal{A}^\mathcal{F}$-$\lim\limits_{k\rightarrow\infty}x_k=\mathcal{L}$ if and only if there is a
subset $G=\{q_1<q_2<...\}$ of $\mathbb{N}$ such that $\delta_\mathcal{A}(G)=1$ and $\mathcal{F}$-$\lim\limits_{k\rightarrow\infty}x_{q_k}=\mathcal{L}$.
\end{thm}
\begin{proof}
Let us assume that $st_\mathcal{A}^\mathcal{F}\mbox{-}\lim\limits_{k\rightarrow\infty}x_k=\mathcal{L}$. Then for each $t\in\mathbb{N}$, let
$$E_t=\{k\in \mathbb{N}:d_L(\mathcal{F}_{x_k\mathcal{L}},\varepsilon_0)\geq\frac{1}{t}\}$$
and
$$G_t=\{k\in \mathbb{N}:d_L(\mathcal{F}_{x_k\mathcal{L}},\varepsilon_0)<\frac{1}{t}\}.$$

Then we have $\delta_\mathcal{A}(E_t)=0$. Also by construction of $G_t$ for each $t\in \mathbb{N}$ we have
$G_1\supset G_2\supset G_3\supset...\supset G_m\supset G_{m+1}\supset...$ with $\delta_\mathcal{A}(G_t)=1$ for each $t\in\mathbb{N}$.

Let $u_1\in G_1$. As $\delta_\mathcal{A}(G_2)=1$, so there exists $u_2\in G_2$ with $u_2>u_1$ such that for
each $n\geq u_2$, $\sum\limits_{k\in G_2}a_{nk}>\frac{1}{2}$.

Again, as $\delta_\mathcal{A}(G_3)=1$, so there exists $u_3\in G_3$ with $u_3>u_2$ such that for each
$n\geq u_3$, $\sum\limits_{k\in G_3}a_{nk}>\frac{2}{3}$.

Thus continuing the above process we get a strictly increasing sequence $\{u_t\}_{t\in\mathbb{N}}$ of
positive integers such that $u_t\in G_t$ for each $t\in\mathbb{N}$ and
$$\sum\limits_{k\in G_t}a_{nk}>\frac{t-1}{t}, \hspace{1in} \text{for each}~n\geq u_t, t\in\mathbb{N}.$$

We now define a set $G$ as follows
$$G=\biggl\{k\in\mathbb{N}:k\in[1,u_1]\biggr\}\bigcup\biggl\{\bigcup\limits_{t\in\mathbb{N}}\{k\in\mathbb{N}:
k\in[u_t,u_{t+1}]~\text{and}~k\in G_t\}\biggr\}.$$

Then, for each $n$, $u_t\leq n<u_{t+1}$, we have

$$\sum\limits_{k\in G}a_{nk}\geq\sum\limits_{k\in G_t}a_{nk}>\frac{t-1}{t}.$$

Therefore $\delta_\mathcal{A}(G)=1$.

Let $\eta>0$. We choose $l\in\mathbb{N}$ such that
$\frac{1}{l}<\eta$. Let $n\geq u_l$, $n\in G$. Then there exists a
natural number $r\geq l$ such that $u_r\leq n<u_{r+1}$. Then by
the construction of $G$, $n\in G_r$. So,
$$d_L(\mathcal{F}_{x_n\mathcal{L}},\varepsilon_0)<\frac{1}{r}\leq\frac{1}{l}<\eta.$$

Thus $d_L(\mathcal{F}_{x_n\mathcal{L}},\varepsilon_0)<\eta$ for each $n\in G,~n\geq u_l$. Hence
$\mathcal{F}\mbox{-}\lim\limits_{\stackrel{\stackrel{k\rightarrow\infty}{k\in G}}~}x_k=\mathcal{L}$. Writing
$G=\{q_1<q_2<...\}$ we have $\delta_\mathcal{A}(G)=1$ and $\mathcal{F}\mbox{-}\lim\limits_{n\rightarrow\infty}x_{q_n}=\mathcal{L}$.

Conversely, let there exists a subset $G=\{q_1<q_2<...\}$ of $\mathbb{N}$ such
that $\delta_\mathcal{A}(G)=1$ and $\mathcal{F}\mbox{-}\lim\limits_{n\rightarrow\infty}x_{q_n}=\mathcal{L}(\in X)$. Then
for each $t>0$, there is an $N_0\in\mathbb{N}$ so that
$$ \mathcal{F}_{x_{q_n}\mathcal{L}}(t)>1-t, \hspace{1in} \forall~ n\geq N_0,$$
i.e.,
$$d_L(\mathcal{F}_{x_{q_n}\mathcal{L}},\varepsilon_0)<t, \hspace{1in} \forall~ n\geq N_0.$$
Let $E_t=\{n\in
\mathbb{N}:d_L(\mathcal{F}_{x_{q_n}\mathcal{L}},\varepsilon_0)\geq
t\}$. Then $E_t\subset\mathbb{N} \setminus \{q_{ _{N_0+1}},q_{
_{N_0+2}},...\}$. Now $\delta_\mathcal{A}(\mathbb{N} \setminus \{q_{
_{N_0+1}}, q_{ _{N_0+2}},...\}) = 0$ and so $\delta_\mathcal{A}(E_t)=0$.

Therefore,
$st_\mathcal{A}^\mathcal{F}\mbox{-}\lim\limits_{k\rightarrow\infty}x_k=\mathcal{L}$.
\end{proof}

\begin{thm}\label{thm68}
Let $(X,\mathcal{F},\tau)$ be a PM space and $x=\{x_k\}_{k \in \mathbb{N}}$ be a sequence in $X$. Then
$x_k\xrightarrow{st_\mathcal{A}^\mathcal{F}}\mathcal{L}$ if and only if there exists a
sequence $\{g_k\}_{k\in\mathbb{N}}$ such that $x_k=g_k$ for a.a.$k(\mathcal{A})$ and $g_k\xrightarrow{\mathcal{F}}\mathcal{L}$.
\end{thm}
\begin{proof}
Let $x_k\xrightarrow{st_\mathcal{A}^\mathcal{F}}\mathcal{L}$. Then
by Theorem \ref{thm67}, there is a set
$G=\{q_1<q_2<...<q_n<...\}\subset\mathbb{N}$ such that
$\delta_\mathcal{A}(G)=1$ and
$\mathcal{F}\mbox{-}\lim\limits_{n\rightarrow\infty} x_{q_n} =
\mathcal{L}$.

We now define a sequence $\{g_k\}_{k\in\mathbb{N}}$ as follows:
\[ g_{k} = \left\{
  \begin{array}{l l}
    x_k, & \quad \text{if}~ k\in G \\
    \mathcal{L}, & \quad \text{if}~ k\notin G.
  \end{array} \right.\]\\
Then clearly, $g_k\xrightarrow{\mathcal{F}}\mathcal{L}$ and also
$\delta_\mathcal{A}(\{k\in\mathbb{N}:x_k\neq g_k\})=0$ i.e., $x_k=g_k$ for $a.a.k(\mathcal{A})$.

Conversely, there exists a sequence $\{g_k\}_{k\in\mathbb{N}}$ such that $x_k=g_k$ for a.a.$k(\mathcal{A})$ and
$g_k\xrightarrow{\mathcal{F}}\mathcal{L}$. Let $t>0$ be given. Since $\mathcal{A}$ is a nonnegative
regular summability matrix so there exists an $N_0\in\mathbb{N}$ such that for each $n\geq N_0$, we get
$$\sum\limits_{x_k\notin \mathcal{N}_\mathcal{L}(t)}a_{nk}\leq\sum\limits_{x_k\neq g_k}a_{nk}+
\sum\limits_{g_k\notin \mathcal{N}_\mathcal{L}(t)}a_{nk}.$$

As $\{g_k\}_{k\in\mathbb{N}}$ is strongly convergent to $\mathcal{L}$, so the set
$\{k\in\mathbb{N}:g_k\notin\mathcal{N}_\mathcal{L}(t)\}$ is finite and
so $\delta_\mathcal{A}(\{k\in \mathbb{N}:g_k\notin\mathcal{N}_\mathcal{L}(t)\})=0$.\\
Thus,
\begin{eqnarray*}
&& ~~~\delta_\mathcal{A}(\{k\in \mathbb{N}:x_k\notin\mathcal{N}_\mathcal{L}(t)\})\\
&& \leq \delta_\mathcal{A}(\{k\in \mathbb{N}:x_k\neq g_k\})+ \delta_\mathcal{A}(\{k\in \mathbb{N}:g_k\notin\mathcal{N}_\mathcal{L}(t)\})=0.
\end{eqnarray*}
Hence, $\delta_\mathcal{A}(\{k\in \mathbb{N}:x_k\notin\mathcal{N}_\mathcal{L}(t)\})=0$. Therefore, the
sequence $\{x_k\}_{k\in\mathbb{N}}$ is strongly $\mathcal{A}$-statistically convergent to $\mathcal{L}$.
\end{proof}

%---------------------------Definition 3.2------------------------------------------
We now study some properties of strong $\mathcal{A}$-statistical
Cauchyness in a PM space. For this we first prove the following
lemma in a metric space.

%---------------------------Lemma 3.6------------------------------------------
\begin{lem}\label{lem61}
Let $(X,\rho)$ be a metric space and $x=\{x_k\}_{k\in\mathbb{N}}$ be a sequence in $X$. Then the following statements are equivalent:
\begin{enumerate}
\item $x$ is an $\mathcal{A}$-statistically Cauchy sequence. \item
For all $\gamma>0$, there is a set $\mathcal{M}\subset\mathbb{N}$
such that $\delta_\mathcal{A}(\mathcal{M})=0$ and
$\rho(x_m,x_n)<\gamma$ for all $m,n\notin \mathcal{M}$. \item For
every $\gamma>0$,
$\delta_\mathcal{A}(\{j\in\mathbb{N}:\delta_\mathcal{A}(D_j(\gamma))\neq
0\})=0$, where $D_j(\gamma)=\{k\in\mathbb{N}:\rho(x_k,x_j)\geq
\gamma\}$, $j\in\mathbb{N}$.
\end{enumerate}
\end{lem}

\begin{thm}\label{thm69}
Let $(X,\mathcal{F},\tau)$ be a PM space and $x=\{x_k\}_{k\in\mathbb{N}}$ be a sequence in $X$. If $x$ is
strongly $\mathcal{A}$-statistically convergent, then $x$ is strong $\mathcal{A}$-statistically Cauchy.
\end{thm}

\begin{proof}
The proof directly follows from Theorem 3.5 \cite{Se2}, by taking
$\mathcal I = \mathcal{I}_{\mathcal A}$.
\end{proof}

\begin{thm}\label{thm610}
Let $(X,\mathcal{F},\tau)$ be a PM space and $x=\{x_k\}_{k\in\mathbb{N}}$ be a sequence in $X$. If the sequence $x$ is
strong $\mathcal{A}$-statistically Cauchy, then for each $t>0$, there is a set $\mathcal{P}_t\subset \mathbb{N}$ with
$\delta_\mathcal{A}(\mathcal{P}_t)=0$ such that $\mathcal{F}_{x_mx_j}(t)>1-t$ for any $m,j\notin \mathcal{P}_t$.
\end{thm}

\begin{proof}
Let $x$ be strong $\mathcal{A}$-statistically Cauchy. Let $t>0$. Then there exists a $\gamma=\gamma(t)>0$ such that,
$$\mathcal{F}_{lr}(t)>1-t ~\text{whenever}~ \mathcal{F}_{lj}(\gamma)>1-\gamma ~\text{and}~ \mathcal{F}_{jr}(\gamma)>1-\gamma.$$

As $x$ is strong $\mathcal{A}$-statistically Cauchy, so there is an $k_0=k_0(\gamma)\in\mathbb{N}$ such that
$$\delta_\mathcal{A}(\{k\in \mathbb{N}:\mathcal{F}_{x_kx_{k_0}}(\gamma)\leq 1-\gamma\})=0.$$
Let $\mathcal{P}_t=\{m\in \mathbb{N}:\mathcal{F}_{x_mx_{k_0}}(\gamma)\leq 1-\gamma\}$. Then $\delta_\mathcal{A}(\mathcal{P}_t)=0$
and $\mathcal{F}_{x_mx_{k_0}}(\gamma)>1-\gamma$ and $\mathcal{F}_{x_jx_{k_0}}(\gamma)>1-\gamma$ for $m,j\notin \mathcal{P}_t$.
Hence for every $t>0$, there is a set $\mathcal{P}_t\subset\mathbb{N}$ with $\delta_\mathcal{A}(\mathcal{P}_t)=0$ such
that $\mathcal{F}_{x_mx_j}(t)>1-t$ for every $m,j\notin \mathcal{P}_t$.
\end{proof}

\begin{cor}\label{cor61}
Let $(X,\mathcal{F},\tau)$ be a PM space and $x=\{x_k\}_{k\in\mathbb{N}}$ be a sequence in $X$. If $x$ is
strong $\mathcal{A}$-statistically Cauchy, then for each $t>0$, there is a set $\mathcal{Q}_t\subset \mathbb{N}$ with
$\delta_\mathcal{A}(\mathcal{Q}_t)=1$ such that $\mathcal{F}_{x_mx_j}(t)>1-t$ for any $m,j\in \mathcal{Q}_t$.
\end{cor}

\begin{thm}\label{thm611}
Let $(X,\mathcal{F},\tau)$ be a PM space and $x=\{x_k\}_{k\in\mathbb{N}}$, $y=\{y_k\}_{k\in\mathbb{N}}$ be
two strong $\mathcal{A}$-statistically Cauchy sequences in $X$. Then $\{\mathcal{F}_{{x_k}{y_k}}\}_{k\in\mathbb{N}}$ is
an $\mathcal{A}$-statistically Cauchy sequence in $(\mathcal{D}^+,d_L)$.
\end{thm}

\begin{proof}
As $x$ and $y$ are strong $\mathcal{A}$-statistically Cauchy
sequences, so by corollary \ref{cor61}, for every $\gamma>0$ there
are $\mathcal{U}_\gamma, \mathcal{V}_\gamma\subset\mathbb{N}$ with
$\delta_\mathcal{A}(\mathcal{U}_\gamma)=\delta_\mathcal{A}(\mathcal{V}_\gamma)=1$,
so that $\mathcal{F}_{x_mx_j}(\gamma)>1-\gamma$ holds for any
$m,j\in \mathcal{U}_\gamma$ and
$\mathcal{F}_{y_ny_z}(\gamma)>1-\gamma$ holds for any $n,z\in
\mathcal{V}_\gamma$. Let
$\mathcal{W}_\gamma=\mathcal{U}_\gamma\cap \mathcal{V}_\gamma$.
Then $\delta_\mathcal{A}(\mathcal{W}_\gamma)=1$. So, for every
$\gamma>0$, there is a set $\mathcal{W}_\gamma\subset\mathbb{N}$
with $\delta_\mathcal{A}(\mathcal{W}_\gamma)=1$ so that
$\mathcal{F}_{x_px_q}(\gamma)>1-\gamma$ and
$\mathcal{F}_{y_py_q}(\gamma)>1-\gamma$ for any $p,q\in
\mathcal{W}_\gamma$. Now let $t>0$. Since $\mathcal{F}$ is
uniformly continuous so there exists a $\gamma(t)>0$ and hence a
set $\mathcal{W}_\gamma=\mathcal{W}_t\subset\mathbb{N}$ with
$\delta_\mathcal{A}(\mathcal{W}_t)=1$ so that
$d_L(\mathcal{F}_{x_py_p},\mathcal{F}_{x_qy_q})<t$ for any $p,q\in
\mathcal{W}_t$. Then the result follows from Lemma \ref{lem61}.
\end{proof}

\section{\textbf{Strong $\mathcal{A}$-statistical limit points and strong $\mathcal{A}$-statistical cluster points}}
In this section we study some basic properties of the notions of
strong $\mathcal{A}$-statistical limit points and strong
$\mathcal{A}$-statistical cluster points of a sequence in a PM
space including their interrelationship.

\begin{defn}\cite{Sh5}\label{def66}
Let $(X,\mathcal{F},\tau)$ be a PM space and
$x=\{x_k\}_{k\in\mathbb{N}}$ be a sequence in $X$. An element
$\mathcal{L}\in X$ is called a strong limit point of $x$, if there
is a subsequence of $x$ that strongly converges to $\mathcal{L}$.
\end{defn}
To denote the set of all strong limit points of any sequence $x$
in a PM space $(X,\mathcal{F},\tau)$ we use the notation
$L_x^\mathcal{F}$.

%---------------------------Definition 4.2------------------------------------------

\begin{defn}\cite{Se2}\label{def67}
Let $(X,\mathcal{F},\tau)$ be a PM space and $\mathcal I$ be an
admissible ideal in $\mathbb{N}$ and $x=\{x_k\}_{k\in\mathbb{N}}$
be a sequence in $X$. An element $\zeta\in X$ is said to be a
strong $\mathcal{I}$-limit point of $x$, if there is a set $B=
\{b_1 < b_2 < ...\} \subset \mathbb{N}$ such that $ B \notin
\mathcal I$ and $\{x_{b_k}\}_{k \in \mathbb{N}}$ strongly
converges to $\zeta$.
\end{defn}

\begin{defn}\cite{Se2}\label{def68}
Let $(X,\mathcal{F},\tau)$ be a PM space and $\mathcal I$ be an
admissible ideal in $\mathbb{N}$ and $x=\{x_k\}_{k\in\mathbb{N}}$
be a sequence in $X$. An element $\nu\in X$ is said to be a strong
$\mathcal{I}$-cluster point of $x$, if for every $t>0$, the set
$\{k\in\mathbb N : x_k\in\mathcal{N}_\nu(t)\}) \notin \mathcal I$.
\end{defn}

\begin{note}
If $ \mathcal I = \mathcal{I}_{\mathcal A} = \{ B \subset
\mathbb{N}: \delta_{\mathcal A} (B) = 0\}$, where $\mathcal A =
(a_{nk})$ is an $\mathbb{N} \times \mathbb{N}$ non negative
regular summability matrix, then the notions of strong
$\mathcal{I}_{\mathcal A}$-limit point and strong
$\mathcal{I}_{\mathcal A}$-cluster point of sequences in a PM
space become the notions of strong $\mathcal A$-statistical limit
point and strong $\mathcal A$-statistical cluster point
respectively. Further, if $\mathcal A$ be the Cesaro matrix $C_1$,
the the notions of strong $\mathcal{I}_{C_1}$-limit point and
strong $\mathcal{I}_{C_1}$-cluster point of sequences in a PM
space become the notions of strong statistical limit point
\cite{Se} and strong statistical cluster point \cite{Se}
respectively.
\end{note}

We denote the set of all strong $\mathcal{A}$-statistical limit
points and strong $\mathcal{A}$-statistical cluster points of a
sequence $x$ in a PM space $(X,\mathcal{F},\tau)$ by
$\Lambda_x^{st}(\mathcal{A})_s^\mathcal{F}$ and
$\Gamma_x^{st}(\mathcal{A})_s^\mathcal{F}$ respectively.

\begin{thm}\label{thm616}
Let $(X,\mathcal{F},\tau)$ be a PM space and
$x=\{x_k\}_{k\in\mathbb{N}}$ be a sequence in $X$. Then
$\Lambda_x^{st}(\mathcal{A})_s^\mathcal{F}\subset
\Gamma_x^{st}(\mathcal{A})_s^\mathcal{F}\subset L_x^\mathcal{F}$.
\end{thm}

\begin{proof}
The proof directly follows from Theorem 4.1 \cite{Se2}, by taking
$\mathcal I = \mathcal{I}_{\mathcal A}$.
\end{proof}

\begin{thm}\label{thm617}
Let $(X,\mathcal{F},\tau)$ be a PM space and
$x=\{x_k\}_{k\in\mathbb{N}}$ be a sequence in $X$. If
$st_\mathcal{A}^{\mathcal{F}}\mbox{-}\lim\limits_{k\rightarrow
\infty}x_k = \mu$, then
$\Lambda_x^{st}(\mathcal{A})_s^\mathcal{F}=\Gamma_x^{st}(\mathcal{A})_s^\mathcal{F}=\{\mu\}$.
\end{thm}

\begin{proof}
Let $st_\mathcal{A}^{\mathcal{F}}\mbox{-}\lim\limits_{k\rightarrow
\infty}x_k = \mu$. So for every $t>0$,
$\delta_\mathcal{A}(\{k\in\mathbb{N}:x_k\in
\mathcal{N}_\mu(t)\})=1$. Therefore,
$\mu\in\Gamma_x^{st}(\mathcal{A})_s^\mathcal{F}$. If possible,
suppose there exists an
$\alpha\in\Gamma_x^{st}(\mathcal{A})_s^\mathcal{F}$ such that
$\alpha\neq\mu$. Then $\mathcal{F}_{\alpha\mu}\neq\varepsilon_0$.
Then there is a $t_1>0$ such that
$d_L(\mathcal{F}_{\alpha\mu},\varepsilon_0)=t_1$. Now for $t_1>0$,
there exists a $\gamma>0$ such that
$d_L(\mathcal{F}_{pq},\varepsilon_0)<\gamma$ and
$d_L(\mathcal{F}_{qr},\varepsilon_0)<\gamma $ imply that
$d_L(\mathcal{F}_{pr},\varepsilon_0)<t_1$. Since,
$\alpha\in\Gamma_x^{st}(\mathcal{A})_s^\mathcal{F}$, for that
$\gamma>0$, $\delta_\mathcal{A}(\mathcal{M})\neq 0$, where
$\mathcal{M}=\{k\in\mathbb{N}: x_k\in
\mathcal{N}_\alpha(\gamma)\}$. Let $\mathcal{K}=\{k\in\mathbb{N}:
x_k\in \mathcal{N}_\mu(\gamma)\}$. As $\mu\neq\alpha$, so
$\mathcal{K}\cap \mathcal{M}=\emptyset$ and so $\mathcal{M}\subset
\mathcal{K}^c$. Since
$st_\mathcal{A}^{\mathcal{F}}\mbox{-}\lim\limits_{k\rightarrow
\infty}x_k= \mu$ so $\delta_\mathcal{A}(\mathcal{K}^c)=0$. Hence
$\delta_\mathcal{A}(\mathcal{M})=0$, which is a contradiction.
Therefore, $\Gamma_x^{st}(\mathcal{A})_s^\mathcal{F}=\{\mu\}$.

Again. as
$st_\mathcal{A}^{\mathcal{F}}\mbox{-}\lim\limits_{k\rightarrow
\infty}x_k= \mu$, so by Theorem \ref{thm68}, we have
$\mu\in\Lambda_x^{st}(\mathcal{A})_s^\mathcal{F}$. Then by Theorem
\ref{thm616}, we get
$\Lambda_x^{st}(\mathcal{A})_s^\mathcal{F}=\Gamma_x^{st}(\mathcal{A})_s^\mathcal{F}=\{\mu\}$.
\end{proof}
%-------------
\begin{thm}\label{thm618}
Let $(X,\mathcal{F},\tau)$ be a PM space,
$x=\{x_k\}_{k\in\mathbb{N}}$ and $y=\{y_k\}_{k\in\mathbb{N}}$ be
two sequences in $X$ such that $\delta_\mathcal{A}(\{k\in
\mathbb{N} : x_k \neq y_k\})=0$. Then
$\Lambda_x^{st}(\mathcal{A})_s^\mathcal{F}=\Lambda_y^{st}(\mathcal{A})_s^\mathcal{F}$
and $\Gamma_x^{st}(\mathcal{A})_s^\mathcal{F}=
\Gamma_y^{st}(\mathcal{A})_s^\mathcal{F}$.
\end{thm}

\begin{proof}
Let $\nu \in \Gamma_x^{st}(\mathcal{A})_s^\mathcal{F}$ and $t>0$
be given. Let $\mathcal{C}=\{k\in\mathbb N: x_k= y_k\}$. Since
$\delta_\mathcal{A}(\mathcal{C})=1$, so
$\delta_\mathcal{A}(\{k\in\mathbb N:x_k\in
\mathcal{N}_\nu(t)\}\cap \mathcal{C})$ is not zero. This gives
$\delta_\mathcal{A}(\{k\in\mathbb{N}:y_k\in
\mathcal{N}_\nu(t)\})\neq 0$ and so $\nu \in
\Gamma_y^{st}(\mathcal{A})_s^\mathcal{F}$. Since $\nu \in
\Gamma_x^{st}(\mathcal{A})_s^\mathcal{F}$ is arbitrary, so
$\Gamma_x^{st}(\mathcal{A})_s^\mathcal{F}\subset
\Gamma_y^{st}(\mathcal{A})_s^\mathcal{F}$. By similar argument we
get $ \Gamma _x^{st}(\mathcal{A})_s^\mathcal{F} \supset
\Gamma_y^{st}(\mathcal{A})_s^\mathcal{F}$. Hence
$\Gamma_x^{st}(\mathcal{A})_s^\mathcal{F}=
\Gamma_y^{st}(\mathcal{A})_s^\mathcal{F}$.

Now let $\mu\in \Lambda_y^{st}(\mathcal{A})_s^\mathcal{F}$. Then
$y$ has an $\mathcal{A}$-nonthin subsequence $\{y_{k_j}\}_{j\in
\mathbb N}$ that strongly converges to $\mu$. Let
$\mathcal{M}=\{k_j \in \mathbb{N} : y_{k_j} =  x_{k_j}\}$. Since
$\delta_\mathcal{A}(\{k_j\in \mathbb{N} : y_{k_j}\neq
x_{k_j}\})=0$ and $\{y_{k_j}\}_{j\in \mathbb{N}}$ is an
$\mathcal{A}$-nonthin subsequence of $y$ so
$\delta_\mathcal{A}(\mathcal{M}) \neq 0$. Now using the set
$\mathcal{M}$ we get an $\mathcal{A}$-nonthin subsequence
$\{x_{k_j}\}_{\mathcal{M}}$ of $x$ that strongly converges to
$\mu$. Thus $\mu \in \Lambda_x^{st}(\mathcal{A})_s^\mathcal{F}$.
As $\mu \in \Lambda_y^{st}(\mathcal{A})_s^\mathcal{F}$ is
arbitrary, so $\Lambda_y^{st}(\mathcal{A})_s^\mathcal{F}\subset
\Lambda_x^{st}(\mathcal{A})_s^\mathcal{F}$. Similarly we have
$\Lambda_x^{st}(\mathcal{A})_s^\mathcal{F}\subset \Lambda
_y^{st}(\mathcal{A})_s^\mathcal{F}$. Therefore
$\Lambda_x^{st}(\mathcal{A})_s^\mathcal{F}=
\Lambda_y^{st}(\mathcal{A})_s^\mathcal{F}$.
\end{proof}

\begin{thm}\label{thm619}
Let $(X,\mathcal{F},\tau)$ be a PM space and
$x=\{x_k\}_{k\in\mathbb{N}}$ be a sequence in $X$. Then the set
$\Gamma_x^{st}(\mathcal{A})_s^\mathcal{F}$ is a strongly closed
set.
\end{thm}

\begin{proof}
The proof directly follows from Theorem 4.2 \cite{Se2}, by taking
$\mathcal I = \mathcal{I}_{\mathcal A}$.
\end{proof}
%---------------------------Theorem 4.5------------------------------------------
\begin{thm}\label{thm620}
Let $(X,\mathcal{F},\tau)$ be a PM space and
$x=\{x_k\}_{k\in\mathbb{N}}$ be a sequence in $X$. Let
$\mathcal{C}$ be a strongly compact subset of $X$ such that
$\mathcal{C}\cap\Gamma_x^{st}(\mathcal{A})_s^\mathcal{F}=\emptyset$.
Then $\delta_\mathcal{A}(\mathcal{M})=0$, where
$\mathcal{M}=\{k\in\mathbb{N}:x_k\in \mathcal{C}\}$.
\end{thm}

\begin{proof}
As
$\mathcal{C}\cap\Gamma_x^{st}(\mathcal{A})_s^\mathcal{F}=\emptyset$,
so for all $\beta\in \mathcal{C}$, there exists a  $t=t(\beta)>0$
so that $\delta_\mathcal{A}(\{k\in\mathbb{N}:x_k\in
\mathcal{N}_\beta(t)\})=0$. Then the family of strongly open sets
$\mathcal{Q}=\{\mathcal{N}_\beta(t):\beta\in \mathcal{C}\}$ forms
a strong open cover of $\mathcal{C}$. As $\mathcal{C}$ is a
strongly compact set, so there exists a finite subcover
$\{\mathcal{N}_{\beta_1}(t_1),
\mathcal{N}_{\beta_2}(t_2),...,\mathcal{N}_{\beta_m}(t_m)\}$ of
the strong open cover $\mathcal{Q}$. Then
$\mathcal{C}\subset\bigcup\limits_{j=1}^m\mathcal{N}_{\beta_j}(t_j)$
and also for each $j=1,2,...,m$ we have
$\delta_\mathcal{A}(\{k\in\mathbb{N}:x_k\in
\mathcal{N}_{\beta_j}(t_j)\})=0$.

Now since $\mathcal{A}$ is a nonnegative regular summability
matrix so there exists an $N_0\in \mathbb{N}$ such that for each
$n\geq N_0$, we get
$$\sum\limits_{x_k\in \mathcal{C}}a_{nk}\leq\sum\limits_{j=1}^m\sum\limits_{x_k\in \mathcal{N}_{\beta_j}(t_j)}a_{nk}.$$ Then we have,\\
\begin{eqnarray*}
&&~~~\lim\limits_{n\rightarrow\infty}\sum\limits_{x_k\in
\mathcal{C}}a_{nk}\leq\sum\limits_{j=1}^m\lim\limits_{n\rightarrow\infty}
\sum\limits_{x_k\in \mathcal{N}_{\beta_j}(t_j)}a_{nk}=0.\\
\end{eqnarray*}
This gives $\delta_\mathcal{A}(\{k\in\mathbb{N}:x_k\in
\mathcal{C}\})=0$.
\end{proof}
%---------------------------Theorem 4.6------------------------------------------
\begin{thm}\label{thm621}
Let $(X,\mathcal{F},\tau)$ be a PM space and
$x=\{x_k\}_{k\in\mathbb{N}}$ be a sequence in $X$. If $x$ has a
strongly bounded $\mathcal{A}$-nonthin subsequence  then the set
$\Gamma_x^{st}(\mathcal{A})_s^\mathcal{F}$ is nonempty and
strongly closed.
\end{thm}

\begin{proof}
Let $\{x\}_\mathcal{M}$ be a strongly bounded
$\mathcal{A}$-nonthin subsequence of $x$. Then
$\delta_\mathcal{A}(\mathcal{M})\neq 0$ and there exists a
strongly compact subset $\mathcal{C}$ of $X$ such that $x_k\in
\mathcal{C}$ for all $k\in \mathcal{M}$. If
$\Gamma_x^{st}(\mathcal{A})_s^\mathcal{F}=\emptyset$, then
$\mathcal{C}\cap\Gamma_x^{st}(\mathcal{A})_s^\mathcal{F}=\emptyset$
and then by Theorem \ref{thm620}, we get
$\delta_\mathcal{A}(\{k\in\mathbb{N}:x_k\in \mathcal{C}\})=0$. But
$\left|\{k\in \mathbb{N}:k\in\mathcal{M}\}\right|\leq
\left|\{k\in\mathbb{N} :x_k\in \mathcal{C}\}\right|$, which gives
$\delta_\mathcal{A}(\mathcal{M})=0$, which contradicts our
assumption. Hence $\Gamma_x^{st}(\mathcal{A})_s^\mathcal{F}$ is
nonempty and by Theorem \ref{thm619},
$\Gamma_x^{st}(\mathcal{A})_s^\mathcal{F}$ is strongly closed.
\end{proof}
%---------------------------Definition 4.3------------------------------------------
\begin{defn}\label{def69}
Let $(X,\mathcal{F},\tau)$ be a PM space and
$x=\{x_k\}_{k\in\mathbb{N}}$ be a sequence in $X$. Then $x$ is
said to be strongly $\mathcal{A}$-statistically bounded if there
exists a strongly compact subset $\mathcal{C}$ of $X$ such that
$\delta_\mathcal{A}(\{k\in\mathbb{N}:x_k\notin \mathcal{C}\})=0$.
\end{defn}

\begin{note}
If $\mathcal A$ be the Cesaro matrix $C_1$, the the notions of
strong ${C_1}$-statistical boundedness in a PM space becomes
strong statistical boundedness \cite{Du1}.
\end{note}

%---------------------------Theorem 4.7------------------------------------------
\begin{thm}\label{thm622}
Let $(X,\mathcal{F},\tau)$ be a PM space and
$x=\{x_k\}_{k\in\mathbb{N}}$ be a sequence in $X$. If $x$ is
strongly $\mathcal{A}$-statistically bounded, then the set
$\Gamma_x^{st}(\mathcal{A})_s^\mathcal{F}$ is nonempty and
strongly compact.
\end{thm}

\begin{proof}
Let $\mathcal{C}$ be a strongly compact set with
$\delta_\mathcal{A}(\mathcal{V})=0$, where
$\mathcal{V}=\{k\in\mathbb{N}:x_k\notin \mathcal{C}\}$. Then
$\delta_\mathcal{A}(\mathcal{V}^c)=1\neq 0$ and so $\mathcal{C}$
contains a bounded $\mathcal{A}$- nonthin subsequence of $x$. So
by Theorem \ref{thm621},
$\Gamma_x^{st}(\mathcal{A})_s^\mathcal{F}$ is nonempty and
strongly closed. We now prove that
$\Gamma_x^{st}(\mathcal{A})_s^\mathcal{F}$ is strongly compact.
For this, we only show that
$\Gamma_x^{st}(\mathcal{A})_s^\mathcal{F}\subset C$. If possible,
let $\alpha\in \Gamma_x^{st}(\mathcal{A})_s^\mathcal{F}\setminus
\mathcal{C}$. As $\mathcal{C}$ is strongly compact, so there is a
$q>0$ such that $\mathcal{N}_\alpha(q)\cap \mathcal{C}=\emptyset$.
So we get
$\{k\in\mathbb{N}:x_k\in\mathcal{N}_\alpha(q)\}\subset\{k\in\mathbb{N}:x_k\notin
\mathcal{C}\}$, which implies
$\delta_\mathcal{A}(\{k\in\mathbb{N}:x_k\in\mathcal{N}_\alpha(q)\})=0$,
a contradiction to our assumption that
$\alpha\in\Gamma_x^{st}(\mathcal{A})_s^\mathcal{F}$. So,
$\Gamma_x^{st}(\mathcal{A})_s^\mathcal{F}\subset \mathcal{C}$.
Therefore $\Gamma_x^{st}(\mathcal{A})_s^\mathcal{F}$ is nonempty
and strongly compact.
\end{proof}

\section{\textbf{Strong statistical $\mathcal{A}$-summability in PM spaces}}

In this section, following the line of Connor \cite{Co5} we
introduce the notions of strong $\mathcal{A}$-summable sequences
and strong statistical $\mathcal{A}$-summable sequences in PM
spaces.

\begin{defn}\label{def64}
Let $(X,\mathcal{F},\tau)$ be a PM space and $x=\{x_k\}_{k \in\mathbb{N}}$ be a sequence
in $X$. Then $x$ is said to be strongly $\mathcal{A}$-summable to $\mathcal{L} \in X$, if for
every $t>0$, there exists a natural number $k_0$ such that
$$\sum\limits_{k=1}^\infty a_{jk}d_L(\mathcal{F}_{x_k\mathcal{L}},\varepsilon_0)< t,\hspace{0.5in} \text{whenever}~ j\geq k_0.$$
In this case we write, $\lim\limits_{j\rightarrow
\infty}\sum\limits_{k=1}^\infty
a_{jk}d_L(\mathcal{F}_{x_k\mathcal{L}},\varepsilon_0)=0$.
\end{defn}

\begin{thm}\label{thm612}
Let $(X,\mathcal{F},\tau)$ be a PM space and $x=\{x_k\}_{k
\in\mathbb{N}}$ be a sequence in $X$. If $x$ is strongly
$\mathcal{A}$-summable to $\mathcal{L}\in X$, then $x$ is strongly
$\mathcal{A}$-statistically convergent to $\mathcal{L}$.
\end{thm}

\begin{proof}
Let $x$ be strongly $\mathcal{A}$-summable to $\mathcal{L}\in X$.
Then $\lim\limits_{j\rightarrow \infty}\sum\limits_{k=1}^\infty
a_{jk}d_L(\mathcal{F}_{x_k\mathcal{L}},\varepsilon_0)=0$. Let
$t>0$ be given. Then for each $j\in\mathbb{N}$,
\begin{eqnarray*}
&~&~~~\sum\limits_{k=1}^\infty a_{jk}d_L(\mathcal{F}_{x_k\mathcal{L}},\varepsilon_0)\\
&=&~\sum\limits_{\stackrel{\stackrel{k\in \mathbb{N}}{d_L(\mathcal{F}_{x_k\mathcal{L}},\varepsilon_0)
\geq t}}~} a_{jk}d_L(\mathcal{F}_{x_k\mathcal{L}},\varepsilon_0)+\sum\limits_{\stackrel{\stackrel{k\in \mathbb{N}}
{d_L(\mathcal{F}_{x_k\mathcal{L}},\varepsilon_0)< t}}~} a_{jk}d_L(\mathcal{F}_{x_k\mathcal{L}},\varepsilon_0)\\
&\geq &~\sum\limits_{\stackrel{\stackrel{k\in
\mathbb{N}}{d_L(\mathcal{F}_{x_k\mathcal{L}},\varepsilon_0)\geq
t}}~} a_{jk}d_L(\mathcal{F}_{x_k\mathcal{L}},\varepsilon_0) \geq
t \sum\limits_{\stackrel{\stackrel{k\in
\mathbb{N}}{d_L(\mathcal{F}_{x_k\mathcal{L}},\varepsilon_0)\geq
t}}~} a_{jk}.
\end{eqnarray*}
Therefore, $\lim\limits_{j\rightarrow
\infty}\sum\limits_{\stackrel{\stackrel{k\in \mathbb{N}}
{d_L(\mathcal{F}_{x_k\mathcal{L}},\varepsilon_0)\geq t}}~}
a_{jk}=0$. So $x$ is strongly $\mathcal{A}$-statistically
convergent to $\mathcal L$.
\end{proof}

\begin{cor}\label{cor62}
Let $(X,\mathcal{F},\tau)$ be a PM space and $x=\{x_k\}_{k
\in\mathbb{N}}$ be a sequence in $X$. If $x$ is strongly
$\mathcal{A}$-summable to $\mathcal{L}\in X$, then $x$ has a
subsequence which is strongly convergent to $\mathcal{L}$.
\end{cor}

\begin{proof}
Directly follows from  Theorem 5.1 and Theorem \ref{thm67}.
\end{proof}

\begin{defn}\label{def65}
Let $(X,\mathcal{F},\tau)$ be a PM space and $x=\{x_k\}_{k \in\mathbb{N}}$ be a sequence in $X$. Then $x$ is
said to be strongly statistically $\mathcal{A}$-summable to $\mathcal{L} \in X$, if for every $t>0$,
$$d(\{j\in \mathbb{N}: \sum\limits_{k=1}^\infty a_{jk}d_L(\mathcal{F}_{x_k\mathcal{L}},\varepsilon_0)\geq t\})=0.$$
In this case we write,
$st$-$\lim\limits_{j\rightarrow \infty}\sum\limits_{k=1}^\infty a_{jk}d_L(\mathcal{F}_{x_k\mathcal{L}},\varepsilon_0)=0$.
\end{defn}

\begin{thm}\label{thm613}
Let $(X,\mathcal{F},\tau)$ be a PM space and $x=\{x_k\}_{k
\in\mathbb{N}}$ be a sequence in $X$. If $x$ is strongly
$\mathcal{A}$-summable to $\mathcal{L}\in X$, then it is strongly
statistically $\mathcal{A}$-summable to $\mathcal{L}$ also.
\end{thm}

\begin{proof}
The proof is very easy, so omitted.
\end{proof}

\begin{thm}\label{thm614}
Let $(X,\mathcal{F},\tau)$ be a PM space and $x=\{x_k\}_{k
\in\mathbb{N}}$ be a sequence in $X$. If $x$ is strongly
$\mathcal{A}$-statistically convergent to $\mathcal{L}\in X$, then
$x$ is strongly $\mathcal{A}$-summable to $\mathcal{L}$ and hence
strongly statistically $\mathcal{A}$-summable to $\mathcal{L}$.
\end{thm}

\begin{proof}
Since $X$ is a PM space, so for all $p,q\in X$ we have
$d_L(\mathcal{F}_{pq},\varepsilon_0)\leq 1$. Now let $x$ be
strongly $\mathcal{A}$-statistically convergent to $\mathcal{L}\in
X$. Let $t>0$ and
$\mathcal{K}=\{k\in\mathbb{N}:d_L(\mathcal{F}_{x_k\mathcal{L}},\varepsilon_0)\geq
t\}$. For each $j\in\mathbb{N}$,
\begin{eqnarray*}
&~&~~~\sum\limits_{k=1}^\infty a_{jk}d_L(\mathcal{F}_{x_k\mathcal{L}},\varepsilon_0)\\
&=&~\sum\limits_{k\in \mathcal{K}} a_{jk}d_L(\mathcal{F}_{x_k\mathcal{L}},\varepsilon_0)+
\sum\limits_{k\notin \mathcal{K}} a_{jk}d_L(\mathcal{F}_{x_k\mathcal{L}},\varepsilon_0)\\
&\leq &~\sum\limits_{k\in \mathcal{K}}
a_{jk}d_L(\mathcal{F}_{x_k\mathcal{L}},\varepsilon_0)+t\sum\limits_{k\notin
\mathcal{K}} a_{jk}
\\
&\leq &~ \sum\limits_{\stackrel{\stackrel{k\in \mathcal{K}}{d_L(\mathcal{F}_{x_k\mathcal{L}},\varepsilon_0)\geq t}}~} a_{jk}+
 t\sum\limits_{k\notin \mathcal{K}} a_{jk}.
\end{eqnarray*}
Now since $x$ is strongly $\mathcal{A}$-statistically convergent
to $\mathcal{L}$ and $\mathcal{A}$ is a nonnegative regular
summability matrix, so
$\lim\limits_{j\rightarrow\infty}\sum\limits_{k=1}^\infty
a_{jk}d_L(\mathcal{F}_{x_k\mathcal{L}},\varepsilon_0)=0$.
Therefore $x$ is strongly $\mathcal{A}$-summable to $\mathcal{L}$
and hence by Theorem \ref{thm613}, $x$ is strongly statistically
$\mathcal{A}$-summable to $\mathcal{L}$.
\end{proof}

\noindent\textbf{Acknowledgment:} The second author is grateful to
Council of Scientific and Industrial Research, India for his
fellowships funding under CSIR-JRF scheme.
\\
%--------------------------------------------------REFERENCES------------------------------

\end{document}